\documentclass[12pt,reqno]{article}

\usepackage[usenames]{color}
\usepackage{amssymb}
\usepackage{graphicx}
\usepackage{amscd}

\usepackage[colorlinks=true,
linkcolor=webgreen,
filecolor=webbrown,
citecolor=webgreen]{hyperref}

\definecolor{webgreen}{rgb}{0,.5,0}
\definecolor{webbrown}{rgb}{.6,0,0}

\usepackage{color}
\usepackage{fullpage}
\usepackage{float}

\usepackage{graphics,amsmath}
\usepackage{amsthm}
\usepackage{amsfonts}
\usepackage{latexsym}
\usepackage{epsf}

\newcommand{\seqnum}[1]{\href{http://oeis.org/#1}{\underline{#1}}}
\def\Enn{{\mathbb{N}}}

\begin{document}

\author{Jean-Paul Allouche\thanks{Author partially supported by the ANR
project ``FAN'' (Fractals et Num\'eration), ANR-12-IS01-0002.} \\
CNRS, IMJ-PRG, UPMC, 4 Pl.~Jussieu, F-75252 Paris Cedex 05, France\\
{\tt jean-paul.allouche@imj-prg.fr}\\
\and
Jeffrey Shallit\thanks{Author partially supported by NSERC.}\\
School of Computer Science, University of Waterloo, \\ 
Waterloo, ON  N2L 3G1,
Canada \\
{\tt shallit@cs.uwaterloo.ca}}

\title{On the subword complexity of the fixed point of
$a \rightarrow aab$, $b \rightarrow b$, and generalizations}

\maketitle

\theoremstyle{plain}
\newtheorem{theorem}{Theorem}
\newtheorem{corollary}[theorem]{Corollary}
\newtheorem{lemma}[theorem]{Lemma}
\newtheorem{proposition}[theorem]{Proposition}

\theoremstyle{definition}
\newtheorem{definition}[theorem]{Definition}
\newtheorem{example}[theorem]{Example}
\newtheorem{conjecture}[theorem]{Conjecture}

\theoremstyle{remark}
\newtheorem{remark}[theorem]{Remark}

\begin{abstract}
We find an explicit closed form for the subword complexity of the
infinite fixed point of the morphism sending $a \rightarrow aab$ and
$b \rightarrow b$.  This morphism is then generalized in three
different ways, and we find similar explicit expressions for the
subword complexity of the generalizations.
\end{abstract}

\section{Introduction}

In this paper we start by considering
a certain morphism $h$ over $\{a,b\}$,
namely, the one where $h(a) = aab$ and $h(b) = b$.  This morphism was
previously studied by the authors and J. Betrema
\cite{Allouche&Betrema&Shallit:1989} and Firicel \cite{Firicel:2011}.

We can iterate $h$ (or any endomorphism)
as follows: set
$h^0 (a)$ and $h^n (a) = h(h^{n-1} (a))$
for $n \geq 1$.    
Note that for the particuar morphism $h$ defined above, we have
$|h^n(a)| = 2^{n+1} - 1$ for $n \geq 0$, a fact that is
easily proved by induction on $n$.

The infinite fixed point of $h$, which we denote
by $h^\omega(a)$
is $\lim_{n \rightarrow \infty} h^n (a)$.  It satisfies
$h( h^\omega(a) ) = h^\omega(a) $.    We also define
${\bf z} = h^\omega(a) = aabaabbaabaabbb \cdots$.

Let $\bf a$ be an infinite word, where ${\bf a} = a_0 a_1 a_2 \cdots$.
We define ${\bf a}[j] = a_j$.   Let $[i..j]$ for integers
$i \leq j-1$ denote the sequence $i, i+1, \ldots, j$.
By a {\it factor\/} of an infinite word we mean a sub-block of the form
$a_i a_{i+1} \cdots a_j$ for $0 \leq i \leq j+1 < \infty$,
which we write as ${\bf a}[i..j]$.   If $i = j+1$ then the resulting
subword is empty.  Sometimes we need to distinguish between a factor
(which is the word itself) and an occurrence of that
factor in $\bf a$ (which is specified by a starting position and length).
The {\it subword complexity\/} of an infinite word $\bf a$
is the function $\rho = \rho_{\bf a}$ that maps a natural number
$n$ to the number of distinct factors of $\bf a$ of length $n$.

In this paper we prove the following exact
formula for $\rho_{\bf z} (n)$:

\begin{theorem}
For $n \geq 0$ we have $\rho_{\bf z} (n) = \sum_{0 \leq i \leq n}
	\min(2^i, n-i+1)$.
\label{thm1}
\end{theorem}
\noindent Previously, upper and lower bounds were given by Firicel
\cite{Firicel:2011}.

The first few values of $\rho_{\bf z} (n)$ are given in
Table~\ref{tab-one}.  It is sequence \seqnum{A006697} in Sloane's
{\it Encyclopedia of Integer Sequences} \cite{Sloane}.

\begin{table}[H]
\begin{center}
\begin{tabular}{c|ccccccccccccccccccccccc}
$n$ &    0 & 1 & 2 & 3 & 4 & 5 & 6 & 7 & 8 & 9 & 10 & 11 & 12 & 13 & 14 & 15 & 16 & 17\\
\hline
$\rho_{\bf z}(n)$ & 1 & 2&  4&  6&  9& 13& 17& 22& 28& 35& 43& 51& 60& 70& 81& 93&106&120&\\
\end{tabular}
\end{center}
\caption{Subword complexity of $\bf z$}
\label{tab-one}
\end{table}

Our method is based on the following factorization theorem for
$\bf z$, which appears in \cite{Allouche&Betrema&Shallit:1989}.
Let $k \geq 2$ be an integer, and define
$\nu_k(n)$ to be the exponent of the largest power of $k$
dividing $n$.  

\begin{theorem}
$$ {\bf z} = \prod_{i \geq 1} a \, b^{\nu_2(i)} = \prod_{i \geq 1} a \, a \,
	b^{\nu_2(i) +1} .$$
\label{one}
\end{theorem}

\begin{remark}
It is interesting to note that function
$n \rightarrow \sum_{0 \leq i \leq n} \min(2^i, n-i+1)$
also counts the maximum
number of distinct factors (of all lengths)
that a binary string of length $n$ can have \cite{Ivanyi:1987,Shallit:1993,
Flaxman&Harrow&Sorkin:2004}.
We do not know any bijective proof of this fact, which we leave as
an open problem for the reader.
\end{remark}

We then generalize the morphism $h$ in three different ways,
and compute the subword complexity of each generalization.

\section{The subword complexity of $\bf z$}

By a $b$-run, we mean a maximal occurrence
of a block of consecutive $b$'s within a word.
Here by ``maximal'' we mean that the block has no $b$'s to either the
left or right.  For example, the word $baabbbaabb$ has three $b$-runs, of
length $1, 3,$ and $2$, respectively.  

Given a factor $w$ of $\bf z$, we call a $b$-run occurrence
in $w$ {\it interior\/} to $w$ if it does not correspond to either
a prefix or suffix of $w$.  For example, in $baabbbaabb$ there is
exactly one interior $b$-run, which is of length $3$.

Given an occurrence of a length-$n$ factor $w$ of $\bf z$,
we define its {\it cover} to be the shortest factor of the form
$\prod_{j \leq i \leq k} a \, a \, b^{\nu_2(i)+1}$ for which
$w$ appears as a factor.  The {\it cover interval} is defined
to be the set $\{ j, j+1, \ldots, k \}$.  We call the integer $j$ (resp., $k$)
the {\it left\/} (resp., {\it right}) {\it edge} of the cover.
For example, the underlined factor below has cover
$aabbaabaabbb$ with left edge $2$ and right edge $4$:
$$ aabaab\underline{baabaab}bbaabaabbaabaabbb \cdots .$$

\begin{lemma}
Let $n \geq 1$.
If a factor of {\bf z} is of length $\geq 2^{n+1}+n-2$, then it must contain
a $b$-run of length at least $n$.
\end{lemma}

\begin{proof}
We consider the longest possible factor $w$ of $\bf z$ having all $b$-runs of
length $< n$.   Such a factor clearly occurs either (a) before the first
$b$-run of length $n$ in
$\bf z$, or (b) between two occurrences of a $b$-run of length $\geq n$
in $\bf z$.    

In case (a), the first $b$-run of length $n$ occurs as a suffix
of $h^n (a)$, which is of length $2^{n+1} - 1$.  So by removing the last
letter we get a factor of length $2^{n+1} - 2$ having no $b$-run of length
$n$.  

In case (b), $w$ has a cover with left edge $\ell$ and right edge $r$,
both of which are divisible by
$2^n$.  All other integers in the cover interval are not divisible by
$2^n$, for if they were, $w$ would have a $b$-run of length $\geq n$.
So $r-\ell = 2^n$.  The longest such $w$ must then be of the form
$w = b^{n-1} h^n (a) b^{-1}$,
and the length of this factor is $2^{n+1} + n - 3$.
(If $x = wa$ is a word, and $a$ is a single letter, then by $x a^{-1}$ we
mean the word $w$.)

\end{proof}

\begin{definition}
Define the function $f$ from $\Enn$ to $\Enn$ as follows:
$$ f(i) = j \text{ for } 2^{j+1}+j-2 \leq i \leq 2^{j+2}+j-2.$$
\end{definition}

The first few values of the function $f$ are given in Table~\ref{t-two}.

\begin{table}[H]
\begin{center}
\begin{tabular}{c|ccccccccccccccccccccccc}
$n$ &    0 & 1 & 2 & 3 & 4 & 5 & 6 & 7 & 8 & 9 & 10 & 11 & 12 & 13 & 14 & 15 & 16 & 17\\
\hline
$f(n)$ & 0 & 0 & 0 & 1 & 1 & 1 & 1 & 1 & 2 & 2 &  2 &  2 &  2 &  2 &  2 &  2 & 2 & 3 
\end{tabular}
\end{center}
\caption{Values of the function $f$}
\label{t-two}
\end{table}

\begin{corollary}
For $n \geq 0$ we have
\begin{itemize}
\item[(a)] every factor of $\bf z$ of length $n$ contains a $b$-run of
length at least $f(n)$;

\item[(b)] at least one factor of $\bf z$ of length $n$ has
longest $b$-run of length exactly $f(n)$;

\item[(c)] the shortest factor of $\bf z$ having two occurrences
of a $b$-run of length $n$ is of length $2^{n+1} + n - 1$.

\end{itemize}
\label{cor7}
\end{corollary}

\begin{proof}

For (c), the shortest factor clearly will start and end with 
$b$-runs of length $n$; otherwise we could remove symbols from the
start or end to get a shorter string with the same property.
So the cover interval begins and ends with integers
divisible by $2^{n-1}$.  The difference between these integers
is therefore at least $2^{n-1}$.  So the cover interval
is $n r^{n-1} (1)$.  The string corresponding to this cover interval
is $b^n h^n(a)$, which of length $2^{n+1} + n - 1$.

\end{proof}

\begin{lemma}
For every factor $w$ of $\bf z$, the
longest $b$-run in $w$ has at most one interior occurrence in $w$.
\end{lemma}

\begin{proof}
Let $b^n$ be the longest $b$-run of $w$, and suppose $w$ has
at least two interior occurrences of $b^n$.
Choose two such occurrences that are separated by
the smallest number of symbols.
By Theorem~\ref{one} these occurrences
must correspond to $b^{\nu_2(i)+1}$ where 
$i \in \{ 2^{n-1} m, 2^{n-1} (m+2) \}$
for some odd number $m$.
Then in between these two $b$-runs there is a $b$-run corresponding
to $i = 2^{n-1} (m+1)$, which (since $m+1$ is even) is of length at
least $n+1$, contradicting the assumption that $b^n$ was the longest
$b$-run in $w$.
\end{proof}

\begin{corollary}
A longest $b$-run in a factor $w$ can have at most three occurrences.
When it does have three,
the occurrences must be a prefix, suffix, and a single interior
occurrence.  In this case the $b$-run must be of the form
$b^n$ for some $n \geq 1$ and the factor must be
$b^n h^{n+1} (a) b^{-1}$, of length $2^{n+2} + n - 2$.
\label{five}
\end{corollary}

\begin{lemma}
If a factor $w$ of $\bf z$ of length $n$ has a $b$-run of length
$> f(n)$, then this run occurs only once in $w$. Furthermore, there is
exactly one such factor $w$ corresponding to the choice of the starting
position of this $b$-run.
\label{lem10}
\end{lemma}

\begin{proof}
First, suppose there were two occurrences of such a run of length
$\geq f(n) + 1$ in $w$.  Then
from Corollary~\ref{cor7} (c), this means that $w$ is of length
at least $2^{f(n) + 2} +  f(n)$.  So $n \geq 2^{f(n)+2} + f(n)$.  
But from the definition of $f$ we have $n \leq 2^{f(n)+2} + f(n) - 2$.
This is a contradiction.

Next, suppose we fix the starting position of a $b$-run of length
$> f(n)$ in $w$.    This $b$-run is either (a) a prefix or suffix of $w$,
or (b) is interior to $w$.  

(a) If this $b$-run is a prefix (resp., suffix) of $w$, it corresponds to a left
(resp., right) edge, divisible by $2^{f(n)}$, of a cover interval.
This fixes the next (resp., previous) $2^{f(n)} -1$ elements of the
cover interval, and so the next (resp., previous)
$|h^{f(n)+1} (a)|$ symbols of $\bf z$ (and hence $w$).  
Thus, including the prefix (resp., suffix), the total number of
symbols determined is of length $f(n) + 1 + 2^{f(n)+2} - 1 =
2^{f(n)+ 2} + f(n)$.   But from the
definition of $f$ we have $n \leq 2^{f(n) + 2} + f(n) -2$.  So all the
symbols of $w$ are determined, and there can only be one such factor.

(b) If this $b$-run is interior to $w$ then, it corresponds to an
element of the cover interval that is exactly divisible by
$2^{f(n)}$.  Then, as in the previous case, the
$2^{f(n)+2} - 1$ symbols both preceding
and following this $b$-run are determined.
Again, this means all the symbols of $w$ are determined, and there
can be only one such factor.

\end{proof}

\begin{corollary}
There are exactly $n-t+1$ factors of $\bf z$ of length $n$ having longest
$b$-run of length $t$, for each $t$ with $f(n) < t \leq n$.
\label{cor11}
\end{corollary}

\begin{proof}
If $t > f(n)$, then from Lemma~\ref{lem10} we know there is exactly
one $b$-run of length $t$ in every factor of length $n$.
Furthermore, there is a unique such factor having a $b$-run of
length $t$ at every possible position, and there are $n-t+1$ possible
positions.  
\end{proof}

The preceding corollary counts all length-$n$ factors having longest $b$-run
of length $> f(n)$.  It remains to count those factors having longest
$b$-run of length equal to $f(n)$.

\begin{definition}
Let the function $g$ be defined as follows:
$$ g(n) = \begin{cases}
      2^t - 1, & \text{if } 2^t + t-3 \leq n \leq 2^t + t-1; \\
2^{t+1}+t-2-n, & \text{if } 2^t + t-1 \leq n \leq 2^{t+1} + t-3 .
\end{cases}
$$
The first few values of the function $g$ are given in Table~\ref{tab3}.

\begin{table}[H]
\begin{center}
\begin{tabular}{c|ccccccccccccccccccccccc}
$n$ &    1 & 2 & 3 & 4 & 5 & 6 & 7 & 8 & 9 & 10 & 11 & 12 & 13 & 14 & 15 & 16 & 17 & 18 & 19 & 20 \\
\hline
$g(n)$ & 1 & 1 & 3 & 3 & 3 & 2 & 1 & 7 & 7 &  7 &  6 &  5 &  4 &  3 &  2 & 1 & 15 & 15 & 15 & 14 \\
\end{tabular}
\end{center}
\caption{Values of the function $g$}
\label{tab3}
\end{table}
\end{definition}

\begin{lemma}
Let $n \geq 1$.  The word $\bf z$ has exactly
$g(n)$ distinct length-$n$ factors with longest $b$-run of length $m = f(n)$.
\label{lem13}
\end{lemma}

\begin{proof}
Let $w$ be a factor of length $n$ of $\bf z$.
If the longest $b$-run of $w$ is of length $m = f(n)$, then
from Corollary~\ref{five} we
know that $w$ itself is a factor of $b^m h^{m+1} (a) b^{-1} = 
	b^m h^m (a) h^m (a)$.
Now $b^m h^m (a) h^m (a)$ is of length $2^{m+2} + m - 2$, so there
are at most $2^{m+2} + m - 1 - n$ positions at which such a factor
could begin.  If $n = 2^{m+1} +m - 2$, then it is easy
to check that the factors of length $n$
starting at the last two possible positions are the same as the first
two; they are both $b^m h^m (a) b^{-1}$ and $b^{m-1} h^m (a)$.
If $n = 2^{m+1} + m-1$, then the factor of length $n$
starting at the last possible position is the same as the first;
they are both $b^m h^m (a)$.  
Otherwise, in these cases and when $n \leq 2^{m+2} + m - 2$, all
the factors are distinct (as can be verified by identifying the
position of the first occurrence of $b^m$).
This gives the result.
\end{proof}

We are now ready to prove Theorem~\ref{thm1}.

\begin{proof}
Totalling the factors described in Corollary~\ref{cor11} and
Lemma~\ref{lem13}, we see that
$$\rho_{\bf z} (n) =  g(n) + \sum_{f(n) < t \leq n} (n-t+1) .$$
We now claim that the right-hand-side equals
$\sum_{0 \leq i \leq n} \min(2^i, n-i+1)$.
To see this, note that for $n = 2^j + j-3$ and $n = 2^j + j-2$
we have $g(f(n)) = 2^j - 1$, while for
$2^j + j - 1  \leq n < 2^{j+1} + j-2$ we have
$g(f(n)) + g(f(n)+1) = 2^{j+1} - 1$.  
\end{proof}

\begin{corollary}
The first difference of the subword complexity of $\bf z$ is
$$\prod_{i \geq 0} [2^i..2^{i+1}] = (1,2,2,3,4,4,5,6,7,8,8,9,\ldots) .$$
\end{corollary}
\noindent This is sequence \seqnum{A103354} in Sloane's {\it On-Line Encyclopedia
of Integer Sequences} \cite{Sloane}.

We can also recover a result of Firicel \cite{Firicel:2010,Firicel:2011}:

\begin{corollary}
There are ${{n^2} \over 2} - n \log_2 n + O(n)$ distinct factors
of length $n$ in ${\bf z}_2$.
\end{corollary}

\begin{remark}
This estimate
was used by Firicel to prove that
${\bf z}$ is not $k$-automatic for any $k \geq 2$.
(The proof in \cite{Allouche&Betrema&Shallit:1989} proved this only
for $k =2$.)
\end{remark}

\begin{remark}
Recall that the (principal branch of the) Lambert function $W$ is defined
for $x \geq -1/e$ by $y = W(x)$ if and only if $x = y e^y$.
Then, for $i \in [0,n]$, we have $2^i \leq n-i+1$ if and only if $i
\leq n+1-W( (\log 2) 2^{n+1})/(\log 2)$. Thus, defining the integer $m$ by
$m:= \lfloor n+1-W( (\log 2) 2^{n+1})/(\log 2)\rfloor$, we get
$$
\rho_{{\bf z}_2}(n) = (2(m+1) - 1) + \frac{(n-m)(n-m+1)}{2} .$$ 
This confirms M. F. Hasler's conjecture about sequence \seqnum{A006697} in 
Sloane's {\it On-Line Encyclopedia of Integer Sequences} \cite{Sloane}.

We also can confirm the conjecture of
V. Jovovic from September 19 2005 that $\bf z$ is
the partial summation of Sloane's sequence
\seqnum{A103354}, and is also equal to $\seqnum{A094913}(n) + 1$.
\end{remark}

\section{The first generalization}

The first and most obvious generalization of the morphism $h$ is
to $h_q$ for $q \geq 2$, where $a \rightarrow a^q b$ and $b \rightarrow b$.
Then $h = h_2$.    Let the fixed point of $h_q$ be
${\bf z}_q = z_q (0) z_q (1) z_q(2) \cdots $.  
Then $z_q(n) = a$ if and only if $n$ has a representation
using the digits $0, 1, \ldots, q-1$ in the system of
Cameron and Wood \cite{Cameron&Wood:1993} using the system of
weights $(q^i-1)/(q-1)$.

\begin{theorem}
For $q \geq 2$ the subword complexity of ${\bf z}_q$ is
$\sum_{0 \leq i \leq n} \min(q^i, n-i+1)$.
\end{theorem}

\begin{proof}
Exactly the same as for $q = 2$.
\end{proof}

\begin{remark}
This result was conjectured in a 1997 email discussion between
the second author and Lambros Lambrou.
\end{remark}

\begin{corollary}
The first difference of the subword complexity of ${\bf z}_q$ is
$$ \prod_{i \geq 0} [q^i..q^{i+1}] =
(1,2,\ldots, q-1, q, q, q+1, \ldots, q^2-1, q^2, q^2, q^2+1, \ldots) .$$
\end{corollary} 
\section{The second generalization}

The classical $q$-ary numeration system represents every non-negative
integer, in a unique way, as sums of the form $\sum_{i \geq 0} a_i
q^i$, where $a_i \in \{ 0, 1, \ldots, q-1 \}$ and only finitely many of
the $a_i$ are nonzero.  In this section, we consider a variation of this
numeration system, where $q^i$ is replaced by $q^i - 1$ and the digit
set is restricted to $\{ 0, 1 \}$.  Of course, in the resulting system,
not every non-negative integer has a representation, so we can consider
the characteristic word ${\bf x}_q = x_q (0) x_q (1)  x_q (2) \cdots$ where
$x_q (i)$ is $1$ if $i$ has a representation and $0$ otherwise.

Note that, if $q$ is a prime power, the infinite word ${\bf x}_q$ is
related to the Carlitz formal power series 
$$ 
\Pi := \prod_{j \geq 1} \left( 1 - 
\frac{X^{q^j} - X}{X^{q^{j+1}} - X} \right) \in {\mathbb F}_q[[X^{-1}]].
$$ 
(see \cite{Allouche:1990} and the references therein).

First, we show how to represent the characteristic sequence
${\bf x}_q$ as the image of a fixed point of a morphism:

\begin{theorem}
Let $q \geq 2$, and
let ${\bf x}_q = x_q (0) x_q (1) x_q (2) \cdots $ be the characteristic
word of those integers having a representation of the form
$\sum_{i \geq 1} \epsilon_i (q^i - 1)$, where $\epsilon_i \in \{ 0, 1 \}$.
Then ${\bf x}_q$ is the coding, under the map
$\tau (a) = 1$ and $\tau(b) = \tau(c) = 0$, of the fixed point of the
morphism
\begin{align*}
a &\rightarrow a b^{q-2} a c^{q(q-2)} b \\
b & \rightarrow b \\
c & \rightarrow c^q .
\end{align*}
\end{theorem}

\begin{remark}
This theorem was obtained in an 1995 email discussion between
the first author and G. Rote.
\end{remark}

\begin{remark}
The expressions for $q > 3$ in the previous theorem correspond to a
transition matrix with dominant eigenvalue $q$. 
The subword complexity of this sequence is not $q$-automatic, as proved in 
\cite{Allouche:1990}. Hence it is
not ultimately periodic.
Using a theorem of F. Durand \cite{Durand:2011},
this implies that the sequence cannot be $k$-automatic
for any $k$ that is multiplicatively independent of $q$.
Hence this sequence cannot be $k$-automatic for any $k$.


\end{remark}

Next, we compute the exact value of the first difference of the complexity
function.

\begin{theorem}
Let $q \geq 3$, and
let $d_q (n) = \rho_{{\bf x}_q} (n+1) -  \rho_{{\bf x}_q} (n)$ for
$n \geq 0$ be the first difference of the complexity function for
${\bf x}_q$.  Then $d_q (n) \in \lbrace 1,2 \rbrace$, and
$$ (d_q (n))_{n \geq 0} = \prod_{i \geq 1} 1^{a_q(i)} 2^{b_q (i)} ,$$
where $a_q (i) = (q-3) q^{i-1} + 2$ and $b_q(i) = q^i - 1$ for
$i \geq 1$.
\end{theorem}

Previously, Firicel \cite{Firicel:2010,Firicel:2011} showed that the
complexity function for $q \geq 3$ is $\Theta(n)$.

Proofs of these two theorems will appear in the final version of this paper.

\section{The third generalization}

We can also generalize our construction in a third way.  Again, we use
$q^i - 1$ as the basis for a numeration system, but now we allow the
digit set to be $\{ 0, 1, \ldots, q-1 \}$.  For $q \geq 2$,
let the infinite word
${\bf y}_q = y_q (0) y_q (1) y_q (2) \cdots$ be the characteristic
sequence of those integers representable in the form
$ \sum_{i \geq 1} a_i (q^i - 1)$ with $a_i \in \{ 0, 1, \ldots, q-1 \}$.

\begin{theorem}
The infinite word ${\bf y}_q$ is the fixed point of the morphism
$1 \rightarrow (1 0^{q-2})^q 0$, $0 \rightarrow 0$. 
\end{theorem}

\begin{theorem}
The first difference of the subword complexity of ${\bf y}_q$ is
the sequence given by 
$$\prod_{i \geq 0} ([q^i..q^{i+1}] \amalg (q-1)) ,$$
where by $w \amalg n$ for $w = a_1 a_2 \cdots a_j$  we mean
$a_1^n a_2^n \cdots a_j^n$.
\end{theorem}

Proofs of these two theorems will appear in the final version of this paper.

\end{document}